\documentclass[10pt]{article}
\usepackage[utf8]{luainputenc}
\usepackage{amsmath, amssymb, amsthm, bbold, mathtools}
\usepackage{tabularx}
\usepackage{xurl}
\usepackage{enumerate,float,subcaption,xcolor,multirow}
\mathtoolsset{centercolon}

\usepackage{hyperref}

\newcommand{\cB}{\mathcal{B}}

\newcommand{\cD}{\mathcal{D}}

\newcommand{\cL}{\mathcal{L}}

\newcommand{\cP}{\mathcal{P}}

\newcommand{\ZZ}{\mathbb{Z}}

\newcommand{\PG}{\mathrm{PG}}

\makeatletter
\newcommand{\gauss}[2]{\genfrac{[}{]}{0pt}{}{#1}{#2}\@ifnextchar\bgroup{_}{}}

\newcommand{\Sym}{\mathrm{Sym}}

\newcommand{\0}{\color{black!20}0}

\newcommand{\V}{V(\Gamma)}

\newcommand{\1}{\mathbb{1}}

\newcommand{\<}{\langle}
\renewcommand{\>}{\rangle} 

\newtheorem{theorem}{Theorem}
\newtheorem{corollary}[theorem]{Corollary}
\newtheorem{lemma}[theorem]{Lemma}
\newtheorem{prop}[theorem]{Proposition}
\theoremstyle{definition}

\newtheorem{example}[theorem]{Example}
\newtheorem{remark}[theorem]{Remark}

\title{Design switching on graphs}
\author{Ferdinand Ihringer and Robin Simoens}
\date{}

\begin{document}

\maketitle

\begin{abstract}
We show that each $(r, \lambda)$-design yields a class of switching methods that can be used to produce cospectral graphs. We use this to explain several specific switching methods such as Godsil-McKay (GM) switching and Wang-Qiu-Hu (WQH) switching.
\end{abstract}

\section{Introduction}


One of the central problems in spectral graph theory is to decide whether a graph is determined by its spectrum.
Haemers' conjecture \cite{HaemersDam2002,Haemers2016} says that almost all graphs are determined
by their spectrum. This claim is backed up by several computer experiments \cite{BrouwerSpence,WangWang},
but a proof seems out of reach. In contrast to Haemers' conjecture for the class of all graphs, we know that, for particular types of graphs, almost all of them (within that class) are not determined by their spectrum. Examples are trees \cite{GM,Schwenk}, strongly regular graphs \cite{FonDerFlaas} and cographs \cite{WangHuang}.

A very successful technique in this context has been the method of \textit{switching}.
Informally speaking, a switching for graph is a local modification of the graph from which we obtain a new graph with some properties left invariant. We are concerned
with switchings which produce \textit{cospectral graphs}: graphs with the same spectrum. Two of the most successful
techniques in this context are \textit{Godsil-McKay (GM) switching} \cite{GM} and \textit{Wang-Qiu-Hu (WQH) switching} \cite{WQH}. Recently, more
switchings in this vibe have been defined, for instance \textit{Abiad-Haemers (AH) switching} \cite{AH}. The first goal of this work is to present a unified framework for these switching techniques.

We use the language of \((r,\lambda)\)-designs.
An \emph{\((r,\lambda)\)-design} is an incidence structure \((\cP,\cB)\) such that
\begin{itemize}
    \item every point is contained in exactly \(r\) blocks, and
    \item every pair of points is included in exactly \(\lambda\) blocks.
\end{itemize}
Let \(J_v\) denote the $v\times v$ all-one-matrix.
Our main result is the following.

\begin{theorem}[Design switching]
\label{thm:main}
    Consider two \((r,\lambda)\)-designs \(\cD_1\) and \(\cD_2\) on \(v\) points and \(b\) blocks, whose respective incidence matrices $N_1$ and $N_2$ satisfy \(N_1^TN_1=N_2^TN_2\).
    Define $R=\frac1{r-\lambda}\left(N_1N_2^T-\lambda J_v\right)$.
    Let \(\Gamma\) be a graph with a subgraph \(C\) whose \(v\) vertices are identified both with the points of \(\cD_1\) and those of \(\cD_2\).
    Suppose that:
    \begin{enumerate}[(i)]
        \item The adjacency matrix $A_C$ of $C$ has the property that $R^TA_CR$ is an adjacency matrix.
        \item For every vertex \(x\notin C\), its neighbours in \(C\) form a block of \(\cD_1\).
    \end{enumerate}
    Make every vertex \(x\notin C\) that is adjacent to the points of the \(i\)th block of \(\cD_1\), adjacent to the points of the \(i\)th block of \(\cD_2\). Replace \(C\) by the subgraph with adjacency matrix $R^T A_C R$. The resulting graph is \(\mathbb{R}\)-cospectral with \(\Gamma\).
\end{theorem}

Let us comment on this result. 
\begin{remark}\label{remark:main}
    \begin{enumerate}[(1)]
        \item We allow designs with repeated blocks (non-simple designs) and blocks may have different sizes. In particular, one can always add the empty block and the block containing all points to both designs, without changing the matrix \(R\). Hence a vertex \(x\notin C\) may also be adjacent to all vertices of \(C\) or no vertices of \(C\). Adding the complement of every block to the design, preserves \(R\), so condition (ii) can be relaxed to:
        \begin{enumerate}
            \item[(ii')] Every vertex \(x\notin C\) is adjacent to
            \begin{enumerate}[(a)]
                \item all vertices of \(C\), or
                \item no vertices of \(C\), or
                \item the vertices of \(C\) contained in a block of \(\cD_1\), or
                \item the vertices of \(C\) not contained in a block of \(\cD_1\).
            \end{enumerate}
        \end{enumerate}
        \item The condition \(N_1^TN_1=N_2^TN_2\) on the incidence matrices can be stated combinatorially as follows: the blocks are ordered in such a way that \(|B_{1i}\cap B_{1j}|=|B_{2i}\cap B_{2j}|\) for any two (possibly the same) blocks \(B_{1i}\) and \(B_{1j}\) of \(\cD_1\) and their corresponding blocks \(B_{2i}\) and \(B_{2j}\) of \(\cD_2\). In particular, \(|B_{1i}|=|B_{2i}|\) for any two corresponding blocks \(B_{1i}\) of \(\cD_1\) and \(B_{2i}\) of \(\cD_2\).

        Note that this is also a necessary condition: in general, for a switching method to work, any two vertices that are not contained in $C$ must have the same number of common neighbours in $C$ before and after switching. Algebraically, this can be explained by: \(x^Ty=(R^Tx)^T(R^Ty)\) if \(R\) is orthogonal.
        \item Condition (i) is automatically fulfilled when the induced subgraph on $C$ is edgeless or complete. If so, then after switching it is edgeless or complete again.
    \end{enumerate}
\end{remark}

A special case of Theorem~\ref{thm:main} occurs when \(\cD_1=\cD_2\) and the blocks are permuted.

\begin{corollary}\label{cor:samedesign}
    Let \(\Gamma\) be a graph with a subgraph \(C\) whose vertices are identified with points of an \((r,\lambda)\)-design with incidence matrix \(N\). Let \(\pi\) be a permutation of the blocks such that for all blocks \(B_i\) and \(B_j\), \[|B_i\cap B_j|=|\pi(B_i)\cap\pi(B_j)|.\] Define $R= \frac1{r-\lambda}\left(N(N^\pi)^T-\lambda J_v\right)$, where \(N^\pi\) is obtained from \(N\) by replacing the \(i\)th column by the \(\pi(i)\)th column, for every \(i\). Suppose that:
    \begin{enumerate}[(i)]
        \item The adjacency matrix $A_C$ of $C$ has the property that $R^T A_C R$ is an adjacency matrix.
        \item For every vertex \(x\notin C\), its neighbours in \(C\) form a block.
    \end{enumerate}
    Make every vertex \(x\notin C\) that is adjacent to the points of \(B\), adjacent to the points of \(\pi(B)\). Replace \(C\) by the subgraph with adjacency matrix $R^T A_C R$. The resulting graph is \(\mathbb{R}\)-cospectral with \(\Gamma\).
\end{corollary}

Although we recover many of the previously known switching methods, not all methods can be obtained in this way. In Section~\ref{sec:notallswitchings}, we present a new switching method that does not fall under this unification.

There are plenty of constructions of strongly regular graphs from known strongly regular graphs by local modifications \cite{BIK,I,VanDamKoolen,Wallis}.
Often, ``linear algebra'' explanations in terms of switching make constructions more elegant and the underlying general principles clearer than combinatorial proofs. For instance, compare \cite{I} and \cite{IM} or \cite{VanDamKoolen} and \cite{MunemasaTwisted}.
Our result has the right generality to do so for other cases as well. In particular, we give a linear algebra explanation for why the switching method due to Kantor \cite{BIK,KAlone} gives cospectral graphs. There is also a clear connection with the concept of a \textit{trade} in a design, see \cite{Billington2003,KMT,Ostergard}, which can be seen as a switching method to construct new designs. Using Theorem~\ref{thm:main}, we give some examples of how to construct cospectral graphs using trades in Section~\ref{sec:appl}.\\

This paper is structured as follows. In Section \ref{sec:prelim} we give preliminaries
before proving Theorem~\ref{thm:main} in Section \ref{sec:main}.
In Section \ref{sec:ana} we analyse design switching in detail for various particular designs.
In particular, we recover GM-, WQH-, and AH-switching.
In Section \ref{sec:notallswitchings} we show that not all cospectrality results for graphs
can be explained using design switching.
In Section \ref{sec:appl}, we give a few applications for design switching.
In particular, we give the announced spectral explanation for Kantor's switching method in \cite{BIK}.


\section{Preliminaries} \label{sec:prelim}

Let $O_n$, $I_n$ and $J_n$ denote the $n\times n$ null matrix, identity matrix and all-one matrix respectively. If the size is clear, indices are omitted. A square matrix \(Q\) is \emph{orthogonal} if \(QQ^T=I\). An orthogonal matrix \(Q\) is \emph{regular} if \(QJ=J\). A square matrix is \emph{decomposable} if, after permutations of its rows and columns, it can be written as a block diagonal matrix with at least two non-identity blocks. The \emph{level} of a matrix \(Q\) is the smallest integer \(\ell\) such that \(\ell Q\) is integral, or \(\infty\) if \(Q\) is irrational.

The \emph{spectrum} of a graph is the multiset of eigenvalues of its adjacency matrix. Two graphs are \emph{cospectral} if they have the same spectrum. Two graphs with adjacency matrices \(A\) and \(A'\) are \emph{\(\mathbb{R}\)-cospectral} if \(A+rJ\) and \(A'+rJ\) have the same spectrum for all \(r\in\mathbb{R}\). Two graphs are \(\mathbb{R}\)-cospectral if and only if they are cospectral and their complements are cospectral, if and only if their adjacency matrices are conjugated by a regular orthogonal matrix \cite{JN}.

A \emph{strongly regular graph with parameters $(v, k, \lambda, \mu)$} is a $k$-regular graph (not complete and not edgeless) with $v$ vertices such that each pair of adjacent
vertices has precisely $\lambda$ common neighbours and each pair of non-adjacent vertices has precisely $\mu$ common neighbours. If two strongly regular graphs have the same parameters, then they are cospectral \cite{BvM}.

In this work, a \emph{switching method} is a local graph transformation, resulting in a cospectral graph. It needs a \emph{switching set}, a subset of the vertex set of the graph, together with some conditions. The size of the switching set should be constant in order to keep the operation local. Recall that two real symmetric matrices are cospectral if and only if they are conjugated by an orthogonal matrix. Thus in the case of a switching method, we want the transformation to correspond to a conjugation with an orthogonal matrix \[Q=Q(R)=\begin{pmatrix}R&O\\O&I\end{pmatrix},\] where the order of \(R\) is a constant, while the order of the graph (and thus the matrix \(Q\)) can be made arbitrarily large.

A switching method is \emph{reducible} to certain other methods if it can be obtained by a sequence of them.
Algebraically, a switching method that corresponds to a conjugation of the adjacency matrix with \(Q\) is reducible to the switching methods corresponding to conjugations with \(Q_1,\dots,Q_s\) if \(Q=Q_1\cdots Q_s\) and, for every adjacency matrix \(A\) such that \(Q^TAQ\) is again an adjacency matrix and for every \(i<s\), the matrix \((Q_1\cdots Q_i)^TA(Q_1\cdots Q_i)\) is an adjacency matrix. The notion of reducibility was introduced in \cite[Definition~10]{ABS} for the special case when \(Q\) has level 2.

An \emph{incidence structure} is a tuple \((\cP,\cB)\) where \(\cP\) is a set and \(\mathcal{B}\) is a multiset of subsets of \(\mathcal{P}\). A \emph{point} is an element of \(\cP\) and a \emph{block} is an element of \(\cB\). If \(\cB\) is a set, then the incidence structure is \emph{simple}. The \emph{incidence matrix} of an incidence structure is a matrix \(N\) indexed by its points and blocks as rows and columns respectively, such that \((N)_{pB}=1\) if \(p\in B\) and \((N)_{pB}=0\) if \(p\notin B\).

An \emph{\((r,\lambda)\)-design} is an incidence structure \((\cP,\cB)\) such that
\begin{itemize}
    \item every point is contained in exactly \(r\) blocks, and
    \item every pair of points is included in exactly \(\lambda\) blocks.
\end{itemize}
For more information on \((r,\lambda)\)-designs, see \cite[Section~VI.49]{handbook}.
A \emph{balanced incomplete block design} or \emph{$2$-$(v,k,\lambda)$-design} is an \((r,\lambda)\)-design with $v$ points where each block has size $k$.
If $k=r$, the design is \textit{symmetric}, in which case any two blocks intersect in precisely $\lambda$ points.

\section{Proof of the main theorem}\label{sec:main}

We need the following observations for the proof of Theorem~\ref{thm:main}.

\begin{lemma}\label{lemma:design}
    Consider an \((r,\lambda)\)-design $(\cP,\cB)$ with \(v\) points and incidence matrix \(N\).
    \begin{enumerate}[(i)]
        \item \(NN^T=\lambda J_v+(r-\lambda)I_v\).
        \item For every point \(p\in\cP\), we have \(\sum_{\substack{B\in\cB\\ p\in B}}|B|=r+(v-1)\lambda\).
    \end{enumerate}
\end{lemma}
\begin{proof}
    \begin{enumerate}[(i)]
        \item By definition.
        \item By double counting of $\{(q,B)\in\cP\times\cB\colon\,p,q\in B\}$. \qedhere
    \end{enumerate}
\end{proof}

We are now ready to prove the main theorem.

\begin{proof}[Proof of Theorem~\ref{thm:main}]
    Let \(\cP=\{p_1,\dots,p_v\}\) denote the vertex set of \(C\), so \(\cP\) is equal to the point set of both designs. Denote their block multisets by \(\cB_1=\{B_{11},\dots,B_{1b}\}\) and \(\cB_2=\{B_{21},\dots,B_{2b}\}\) such that the columns of \(N_1\) and \(N_2\) are labelled in the same order. 
    From \(N_1^TN_1=N_2^TN_2\) and Lemma~\ref{lemma:design}, it follows
    that \(R^TR=I_v\) and \(RJ_v=J_v\), so $R$ is a regular orthogonal matrix.
    Label the vertices of \(\Gamma\) in such a way that its adjacency matrix $A_1$ and the adjacency matrix of the new graph $A_2$ have block form
    \[A_1=\begin{pmatrix}A_{1C}&V_1\\V_1^T&W\end{pmatrix}\text{ and }A_2=\begin{pmatrix}A_{2C}&V_2\\V_2^T&W\end{pmatrix}\] where \(A_{1C}\)
    and \(A_{2C}\) are the adjacency matrices of the induced subgraph on \(\cP\). We check that \(A_1=Q^TA_1Q\), where $Q = Q(R)$.
    By construction, \(A_{2C}=R^TA_{1C}R\). In order to see that \(V_2=R^TV_1\) is satisfied, choose \(i\in\{1,\dots,n\}\) and let \(\chi\) be the characteristic vector of \(B_{1i}\). It suffices to prove that \(R^T\chi\) is the characteristic vector of \(B_{2i}\). 
    Let \(j\in\{1,\dots,n\}\). Since \(N_1^TN_1=N_2^TN_2\), we have
    \[\left(N_1^T\chi\right)_j=|B_{1i}\cap B_{1j}|=|B_{2i}\cap B_{2j}|,\]
    see also Remark~\ref{remark:main}(2).
    A double counting of
    \[\{(p,B)\in\cP\times\cB_2\colon\, p\neq p_j\in B\neq B_{2i}\text{ and }p\in B_{2i}\cap B\}\]
    implies
    \[\left(N_2N_1^T\chi\right)_j=\sum_{\substack{B\in\cB_2\\ p_j\in B}}|B_{2i}\cap B|=\begin{cases}(|B_{2i}|-1)\lambda+r&\text{ if }p_j\in B_{2i},\\\lambda|B_{2i}|&\text{ if }p_j\notin B_{2i},\end{cases}\]
    and thus
    \[\left(R^T\chi\right)_j=\begin{cases}\frac1{r-\lambda}\left((|B_{2i}|-1)\lambda+r-\lambda|B_{2i}|\right)=1&\text{ if }p_j\in B_{2i},\\\frac1{r-\lambda}\left(\lambda|B_{2i}|-\lambda|B_{2,i}|\right)=0&\text{ if }p_j\notin B_{2i}.\end{cases}\]\qedhere
\end{proof}

\section{Analysis for particular designs} \label{sec:ana}

In this section, we apply Theorem~\ref{thm:main} to several specific designs. We first show that most of the currently known switching methods follow from this theorem. Next, we look at projective designs as well as examples for small designs.

\subsection{Known switching methods}

In this section, we point out that many of the currently known switching methods can be obtained by Corollary~\ref{cor:samedesign}.

\subsubsection{Godsil-McKay switching}\label{sec:GM}

The following result is one of the most productive methods for finding cospectral graphs.

\begin{theorem}[GM-switching \cite{GM}]\label{thm:GM}
    Let \(\Gamma\) be a graph and let \(C_1,\dots,C_t\) be disjoint subsets of its vertex set such that the following hold for all \(i,j\in\{1,\dots,t\}\).
    \begin{enumerate}[(i)]
        \item Every vertex in \(C_i\) has the same number of neighbours in \(C_j\).
        \item Every vertex \(x\notin C_1\cup\dots\cup C_t\) has \(0\), \(\frac{1}{2}|C_i|\) or \(|C_i|\) neighbours in \(C_i\).
    \end{enumerate}
    For all \(i\in\{1,\dots,t\}\) and every vertex \(x\notin C_1\cup\dots\cup C_t\) that has exactly \(\frac{1}{2}|C_i|\) neighbours in \(C_i\), swap the adjacencies between \(x\) and \(C_i\). The resulting graph is \(\mathbb{R}\)-cospectral with \(\Gamma\).
\end{theorem}

Define \(c_i=|C_i|\) for all \(i\in\{1,\dots,t\}\). If we want to specify the size of the switching set, we speak of ``GM$_{c_1+\dots+c_t}$~switching''.
Algebraically, this method corresponds to a conjugation of the adjacency matrix with
the block diagonal matrix \(\mathrm{diag}(R_1,\dots,R_t,I)\), where $$R_i=\frac2{c_i}J_{c_i}-I_{c_i}$$ for all \(i\in\{1,\dots,t\}\).
Let \(\cP=C_1\cup\dots\cup C_t\) and let \(\cB\) be the set of all subsets of \(\cP\) that intersect \(C_i\) in \(0\), \(\frac{1}{2}c_i\) or \(c_i\) elements, for all \(i\in\{1,\dots,t\}\). Together, they form an incidence structure with \(r=\frac12|\cB|=\frac12\prod_{i=1}^t\left(\binom{c_i}{\frac12c_i}+2\right)\) blocks through each point.
If \(t=1\), then the number of blocks through any two points is equal to \(\lambda=\binom{c_1-2}{\frac12c_1-2}+1\).
If \(t\geq2\) and \(|C_i|=4\) for all \(i\in\{1,\dots,t\}\), then the number of blocks through any two points is equal to \(\lambda=2^{3t-2}\).
In both cases, GM-switching is a particular instance of Corollary~\ref{cor:samedesign} when we let \(\pi\) be the permutation that corresponds to the operation in the statement. But in general, we do not know whether for a given set of numbers \((c_1,\dots,c_t)\), there exists an \((r,\lambda)\)-design that gives rise to GM$_{c_1+\cdots+c_t}$ switching.

\begin{example}\label{ex:GM64}
    For GM$_{6+4}$ switching, the number of blocks through two vertices of \(C_1\) is \(40\), while there are 44 blocks through a pair of vertices of \(C_2\). Yet still, this problem can be fixed by adding multiplicities to certain blocks. For example, if we count all blocks containing exactly half of the vertices of \(C_1\) with multiplicity \(4\) and the other blocks with multiplicity \(11\), then the number of blocks through every pair of vertices is always \(\lambda=204\) and the number of blocks through a fixed point is always \(r=408\). Therefore, also if \(t=2\), \(|C_1|=6\) and \(|C_2|=4\), we can reformulate GM-switching as a particular case of Corollary~\ref{cor:samedesign}.
\end{example}

\subsubsection{Wang-Qiu-Hu switching}

Another powerful switching technique is the following.

\begin{theorem}[WQH-switching \cite{WQH}]\label{thm:WQH}
    Let \(\Gamma\) be a graph and let \(C_1^{1},C_1^{2},\dots,C_t^{1},C_t^{2}\) be disjoint subsets of its vertex set such that the following hold for all \(i,j\in\{1,\dots,t\}\).
    \begin{enumerate}[(i)]
        \item \(|C_i^{1}| = |C_i^{2}|\).
        \item Let \(\deg(x,C)\) denote the number of neighbours of \(x\) in \(C\). The difference $\deg(x,C^k_j)-\deg(x,C_j^{3-k})$ is constant for each \(k\in\{1,2\}\) and every vertex \(x\in C^k_i\).
        \item Every vertex \(x\notin C_1^{1}\cup C_1^{2}\cup\dots\cup C_t^1\cup C_t^{2}\) is adjacent to
        \begin{enumerate}[(a)]
            \item all vertices of \(C_i^{1}\) and no vertices of \(C_i^{2}\), or
            \item no vertices of \(C_i^{1}\) and all vertices of \(C_i^{2}\), or
            \item the same number of vertices of \(C_i^{1}\) as of \(C_i^{2}\).
        \end{enumerate}
    \end{enumerate}
    For all \(i\in\{1,\dots,t\}\) and every \(x\notin C_1^{1}\cup C_1^{2}\cup\dots\cup C_t^1\cup C_t^{2}\) for which (a) or (b) holds, swap the adjacencies between \(x\) and \(C_i^{1}\cup C_i^{2}\). The resulting graph is \(\mathbb{R}\)-cospectral with \(\Gamma\).
\end{theorem}

Define \(c_i=|C^1_i|=|C^2_i|\) for all \(i\in\{1,\dots,t\}\). If we want to specify the size of the switching set, we speak of ``WQH$_{2c_1+\dots+2c_t}$~switching''.
Algebraically, this method corresponds to a conjugation of the adjacency matrix with the matrix \(\mathrm{diag}(R_1,\dots,R_t,I)\), where \[R_i =\frac1{c_i}\left(\arraycolsep=2.5pt\def\arraystretch{.9}\begin{array}{rr}-J_{c_i}&J_{c_i}\\J_{c_i}&-J_{c_i}\end{array}\right)+I_{2c_i}\] for all \(i\in\{1,\dots,t\}\).

Let \(\cP=C_1^{1}\cup C_1^{2}\cup\dots\cup C_t^{2}\) and let \(\cB\) be the set of all subsets that intersect \(C_i^{1}\cup C_i^{2}\) either in \(C_i^{1}\), in \(C_i^{2}\) or in equally many elements of \(C_i^{1}\) as of \(C_i^{2}\), for all \(i\in\{1,\dots,t\}\). These sets form an incidence structure with \[r=\frac12|\cB|=\frac12\prod_{i=1}^t\left(2+\sum_{j=0}^{c_i}\binom{c_i}{j}^2\right)=\frac12\prod_{i=1}^t\left(2+\binom{2c_i}{c_i}\right)\] blocks through each point, where the last equality is obtained by counting the number of north-east lattice paths between \((0,0)\) and \((c_i,c_i)\).
If \(c_i=4\) for all \(i\in\{1,\dots,t\}\), then WQH-switching is equivalent to GM-switching on sets of four vertices, and hence a particular case of Corollary~\ref{cor:samedesign}, see Section~\ref{sec:GM}.
If \(t=1\), then the number of blocks through two points is in general not constant. However, we can modify the design by adding multiplicities to the blocks as follows. Let \(c=c_1\). The number of blocks through two points that are both in \(C_1^1\) or both in \(C_1^2\) is equal to $$1+\sum_{i=2}^c\binom{c-2}{i-2}\binom{c}{i}=\binom{2c-2}{c}+1,$$ where the equality can be obtained by counting north-east lattice paths between \((0,0)\) and \((c,c-2)\). The number of blocks through two points, one from \(C_1^{1}\) and one from \(C_1^{2}\), is equal to $$\sum_{i=1}^c\binom{c-1}{i-1}^2=\binom{2c-2}{c-1},$$ the number of north-east lattice paths between \((0,0)\) and \((c-1,c-1)\).
If we let the multiplicity of the blocks \(C_i^{1}\) and \(C_i^{2}\) be equal to \(\binom{2c-2}{c-1}-\binom{2c-2}{c}\geq1\), then the number of blocks through every pair of points becomes \(\binom{2c-2}{c-1}\). Note that the number of blocks through each point is also increased, but remains a constant. We have found an \((r,\lambda)\)-design that, together with the permutation of the blocks that corresponds to the operation in the statement, leads to WQH-switching if \(t=1\).
Note that, in contrast with the case \(t=1\) for GM-switching, the design is not totally obvious here. We need repeated blocks.

\subsubsection{Abiad-Haemers switching}\label{sec:AH}

Abiad and Haemers \cite{AH} gave an algebraic description of switching methods that correspond with a conjugation of the adjacency matrix with a regular orthogonal matrix \(Q\) of level 2. There are essentially three different switching methods and they can be described easily by the orthogonal matrix \(R\) such that \(Q = Q(R)\). We refer to \cite{ABS} for a combinatorial description of these methods.

\paragraph{Infinite family}

The \(2m\times2m\) matrix
\[R_{2m}=\frac{1}{2}\mathrm{circulant}(J,O,\dots,O,Y)
\]
where $J$, $O$ and $Y=2I-J$ are square matrices of order $2$,
yields an infinite family of switching methods which we call AH$_{2m}$ switching.
The 01-vectors \(v\) such that \(R_{2m}^Tv\) is another 01-vector, can be obtained from the following lemma.

\begin{lemma}[{\cite[Lemma~15]{phdaida}}]
\label{lemma:phdaida}
    A 01-vector \(v\) has the property that \(R_{2m}^Tv\) is again a 01-vector if and only if the numbers of ones in each of the tuples \((v_1,v_2),\dots,(v_{2m-1},v_{2m})\), is the same modulo 2.
    If the number of ones is always even (0 or 2), then \(R_{2m}^Tv=v\). If it is always equal to 1, then \(R_{2m}^Tv=(v_3,\dots,v_{2m},v_1,v_2)^T\).
\end{lemma}

We can rephrase the switching method that corresponds to \(R_{2m}\) as a particular case of Corollary~\ref{cor:samedesign} for the \((r=2^m,\lambda=2^{m-1})\)-design with point set \(\cP=\{1,\dots,2m\}\) and blocks \(B\subseteq\cP\) such that \(|B\cap\{2i-1,2i\}|\pmod{2}\) is the same for all \(i\in\{1,\dots,m\}\). The permutation \(\pi\) in Corollary~\ref{cor:samedesign} is a cyclic shift in those cases where the number is 1 modulo 2, and the identity otherwise.

\paragraph{Fano switching}

The \(7\times7\) matrix
\[R_\text{Fano}=\frac12\mathrm{circulant}(-1,1,1,0,1,0,0)
\]
yields a switching method that can be obtained by permuting the lines of the Fano plane, see also Example~\ref{ex:fano} below.

\paragraph{Cube switching}

The \(8\times8\) matrix
\[R_\text{cube}=\frac{1}{2}\left(\arraycolsep=2.5pt\def\arraystretch{.9}\begin{array}{rrrr}
-I & I & I & I \\
I & -Z & I & Z \\
I & Z & -Z & I \\
I & I & Z & -Z
\end{array}\right)\]
where $I$, $J$ and $Z=J-I$, are square matrices of order $2$, corresponds to a switching method that can be obtained from the design that is the affine space \(\mathrm{AG}(3,2)\). See also Example~\ref{ex:AG32} below.

\subsection{Projective planes and Singer cycles}

Taking the points of a finite affine or projective geometry together with subspaces of a fixed dimension for blocks
is a rich source for $(r, \lambda)$-designs. In general, when applying Theorem \ref{thm:main},
it is unclear when $R^T A_C R$ is again an adjacency matrix. The only choices for $A_C$ that always work, are $A_C=O$ or $A_C=J-I$.
In some projective planes, there is another choice: a Singer cycle.
A \emph{Singer cycle} of a projective plane is a collineation that permutes the points cyclically. Such a cycle exists if and only if the plane has a cyclic point-line incidence matrix. With some abuse of nomenclature, we could also use the name \emph{Singer cycle} for a subgraph \(C\) when it is a cycle graph with the same ordering of a Singer cycle in the projective designs.

\begin{theorem}\label{thm:singer}
    Consider two projective planes \(\Pi_1\) and \(\Pi_2\) of order \(q\) that admit a Singer cycle. 
    Let \(\Gamma\) be a graph with a cyclic subgraph \(C\) whose vertices are identified both with points of \(\Pi_1\) and \(\Pi_2\) in the same cyclic order.
    Suppose that every vertex \(x\notin C\) is adjacent to the vertices of \(C\) contained in a line, its complement or all or no vertices of \(C\).
    Make every vertex \(x\notin C\) that is (non)adjacent to the points of the \(i\)th line of \(\Pi_1\), (non)adjacent to the points of the \(i\)th line of \(\Pi_2\). The resulting graph is \(\mathbb{R}\)-cospectral with \(\Gamma\).
\end{theorem}
\begin{proof}
    Let \(A=\mathrm{circulant}(0,1,0,\dots,0,1)\) be the adjacency matrix of the cycle graph of order \(v=q^2+q+1\). If \(N\) is a cyclic incidence matrix of a projective plane of order \(q\) with point set \(\cP=\{p_1,\dots,p_v\}\) and line set \(\cL=\{L_1,\dots,L_v\}\), ordered cyclically with indices modulo \(v\), then
    \begin{align*}
        \left(N^TAN\right)_{ij}&=\left|\left\{k\in\ZZ/v\ZZ\colon\,p_k\in L_i\text{ and }p_{k-1}\in L_j\right\}\right|\\&\quad+\left|\left\{k\in\ZZ/v\ZZ\colon\,p_k\in L_i\text{ and }p_{k+1}\in L_j\right\}\right|\\
        &=\left|L_i\cap L_{j+1}\right|+\left|L_i\cap L_{j-1}\right|\\
        &=\begin{cases}q+2&\text{ if }i=j\pm1,\\2&\text{ else,}\end{cases}\\
        &=\left(qA+2J_v\right)_{ij}
    \end{align*}
    so \(N^TAN=qA+2J_v\) and, by duality, \(NAN^T=qA+2J_v\). If we define \(R=\frac1q\left(N_1N_2^T-J_v\right)\) as in Theorem~\ref{thm:main}, where \(N_1\) and \(N_2\) are the incidence matrices of \(\Pi_1\) and \(\Pi_2\) respectively, then from the previous equalities, together with \(JA=AJ=2J\) and Lemma~\ref{lemma:design}, it follows that \(R^TAR=A\). Thus the statement is a corollary of Theorem~\ref{thm:main}.
\end{proof}

\begin{example}
    Given a projective plane \(\Pi=(\cP,\cL)\) of order \(q\) with a Singer cycle \(\pi=(p_1,\dots,p_{q^2+q+1})\), choose a line \(L\in\cL\) and let \(\Omega=\{p_i\in\cP\colon\, p_{-i}\in L\}\). The set \(\Omega\) is an \emph{oval}: a set of \(q+1\) points no three of which are collinear \cite[Theorem~16]{thesishain}. Define \(L_i=\pi^i(L)\) and \(\Omega_i=\pi^i(\Omega)\). The sets \(\Omega_i\), \(i\in\{1,\dots,q^2+q+1\}\), form a projective plane that is isomorphic to \(\Pi\). It has the same point set \(\cP\) and agrees with the cyclic structure. That is, \(\pi\) also induces a collineation of the lines \(\Omega_i\), \(i\in\{1,\dots,q^2+q+1\}\). The switching operation of Theorem~\ref{thm:singer} comes down to replacing (non)adjacencies with \(L_i\) by (non)adjacencies with \(\Omega_i\).
\end{example}

\begin{remark}\label{rem:planesreducible}
    Switching methods coming from Corollary~\ref{cor:samedesign} for projective planes may be reducible to WQH-switching (Theorem~\ref{thm:WQH}) if the adjacency matrix of the switching set allows it. Indeed, every permutation of lines can be written as a product of transpositions of two lines. Applying Corollary~\ref{cor:samedesign} where \(\pi\) is a transposition, is the same operation as applying WQH-switching on the symmetric difference of those lines, since every other line intersects each of the lines (without the intersection) in equally many points. In particular, if the switching set induces an edgeless or complete subgraph, the switching method from Corollary~\ref{cor:samedesign} when applied to a projective plane of order \(q\), is always reducible to WQH$_{2q}$ switching.

    The same can be said about permutations that preserve parallel lines of affine planes: if the adjacency matrix allows it, then the switching method from Corollary~\ref{cor:samedesign} when applied to an affine plane of order \(q\) and a permutation that preserves parallelism, may be reducible to WQH$_{2q}$ switching.
\end{remark}

\subsection{Classification of switching methods from small designs}

In this section, we analyse some examples of switching methods that can be obtained from Corollary~\ref{cor:samedesign} when applied to certain small $(r,\lambda)$-designs. We make the following restrictions.
\begin{itemize}
    \item We only classify methods coming from one design (Corollary~\ref{cor:samedesign}) instead of two designs (Theorem~\ref{thm:main}), to simplify computations.
    \item We only consider \emph{simple} designs. The reason for avoiding the non-simple case (where blocks can be repeated) is explained in Remark~\ref{remark:nonsimple} below.
\end{itemize}

Our general strategy is as follows. First assume that we are given a simple $(r,\lambda)$-design \(\cD\) on $v$ points and $b$ blocks.
Let \(G\leq\Sym(b)\) be the group of permutations of blocks of \(\cD\) that preserve their intersection sizes. Let \(H=\mathrm{Aut}(\cD)\leq\Sym(v)\) be the automorphism group of the design \(\cD\), acting on its blocks. Observe that \(H\leq G\). Denote \(\Gamma_\pi\) for the graph obtained by Corollary~\ref{cor:samedesign}.

\begin{lemma}
    If \(\pi\) and \(\tau\) are in the same double coset of $H\backslash G/H$, then \(\Gamma_\pi\cong\Gamma_\tau\).
\end{lemma}
\begin{proof}
    If \(\pi\) and \(\tau\) are in the same double coset of $H\backslash G/H$, then \(\tau=h_1\pi h_2\) for certain automorphisms \(h_1,h_2\) of \(\cD\). An automorphism of \(\cD\) induces a permutation of the vertices of the switching set and thus an automorphism of the graph. In other words, \(h_1\) is an automorphism of the graph before switching and \(h_2\) is an automorphism of the graph after switching, so \(\Gamma_\pi\cong\Gamma_\tau\).
\end{proof}

So in order to classify the different switching methods that can be obtained by the design \(\cD\), we choose a representative \(\pi\) for each double coset in $H\backslash G/H$. Given the permutation \(\pi\) of the blocks of \(\cD\), we calculate the matrix \(R\) as defined in Corollary \ref{cor:samedesign}. Note that every double coset gives a different matrix \(R\). For each such matrix \(R\), we can then calculate the possible matrices \(A_C\).

For most examples, we use GAP \cite{GAP} to calculate the double coset representatives and the possible matrices \(R\) and \(A_C\). Since the group $G$ is often difficult to calculate, we calculate instead a group that is somewhere between $G$ and $\Sym(b)$, namely the group of permutations of blocks of $\cD$ that preserve a certain fixed intersection size, as this group can be calculated more easily as the automorphism group of the following graph. Its vertices are the blocks of the design. Two blocks are adjacent if they intersect in a given, fixed size. For this automorphism group, we calculate the double cosets and a representative for each of them. Only after that step, we filter down the representatives to those that preserve the intersection size of blocks. To go over some small designs, we use the \texttt{design} package \cite{GAPdesign}.

\begin{example}
\label{ex:fano}
    We illustrate the above approach for the Fano plane PG$(2,2)$, the projective plane of order 2, which is an \((r=3,\lambda=1)\) design on seven points and seven blocks. We can choose its incidence matrix to be
$$N=\begin{pmatrix}
1&1&1&0&0&0&0\\
1&0&0&1&1&0&0\\
1&0&0&0&0&1&1\\
0&1&0&1&0&1&0\\
0&1&0&0&1&0&1\\
0&0&1&1&0&0&1\\
0&0&1&0&1&1&0\\
\end{pmatrix}.$$
Since any two lines intersect in the same number of points (that is, one point), \(G=\Sym(7)\) is the full symmetric group of order seven, acting on the lines. The automorphism group of the design is the collineation group $H=\mathrm{P\Gamma L}(3,2)$ acting on the lines. For the ordering $1,\dots,7$ of the lines (columns) according to the above incidence matrix, we find $H = \<(1~3)(5~7), (1~4~2)(3~5~6)\>$. Now, $H\backslash G/H$ consists of four double cosets with the following representatives:
\begin{align*}
    &\pi_1 = \mathrm{id}, && \pi_3 = (5~6~7), \\
    &\pi_2 = (6~7), && \pi_4 = (3~4)(5~6~7).
\end{align*}
If we denote $R_i=\frac12\left(N(N^{\pi_i})^T-J_7\right)$, following the notation of Corollary~\ref{cor:samedesign}, then
\begin{align*}
 & R_1 = I_7,
 && R_3 = \frac12
 \left(\arraycolsep=2.5pt\def\arraystretch{.9}\begin{array}{rrrrrrr}
  2&  \0&  \0&  \0&  \0&  \0&  \0 \\
  \0&   1&   1&   1&   -1&  0&   0 \\
  \0&   1&   1&  -1&   1&   0&   0 \\
  \0&   0&   0&   1&   1&   1&  -1 \\
  \0&   0&   0&   1&   1&  -1&   1 \\
  \0&   1&  -1&   0&   0&   1&   1 \\
  \0&  -1&   1&   0&   0&   1&   1
 \end{array}\right),\\
 & R_2 = \frac12
 \left(\arraycolsep=2.5pt\def\arraystretch{.9}\begin{array}{rrrrr}
  2I_3 & \0 & \0 & \0 & \0 \\
  \0& 1 & 1 & 1 & -1 \\
  \0& 1 & 1 & -1 & 1 \\
  \0& 1 & -1 & 1 & 1 \\
  \0& -1 & 1 & 1 & 1
 \end{array}\right),
 && R_4 = \frac12
 \left(\arraycolsep=2.5pt\def\arraystretch{.9}\begin{array}{rrrrrrr}
  1&   1&   0&   1&   0&   0&  -1 \\
  1&   0&   1&   0&  -1&   0&   1 \\
  0&   1&   1&  -1&   1&   0&   0 \\
  1&  -1&   0&   0&   1&   1&   0 \\
  0&   0&   0&   1&   1&  -1&   1 \\
  0&   1&  -1&   0&   0&   1&   1 \\
 -1&   0&   1&   1&   0&   1&   0
 \end{array}\right).
\end{align*}
For each of the $R_i$, we determine the adjacency matrices $A_C$ such that $R^TA_CR$ is an adjacency matrix again.
\begin{enumerate}
    \item For $R_1$, any adjacency matrix suffices. But \(\pi_1\) does not do anything: \(\Gamma_{\pi_1}=\Gamma\). So this case is not interesting.
    \item $R_2$ corresponds to GM$_4$ switching (or, equivalently, WQH$_4$ switching) on the last four vertices. This agrees with Remark~\ref{rem:planesreducible}.
    \item $R_3$ corresponds to AH$_6$ switching on the last six vertices. 
    \item $R_4$ corresponds to what is called Fano switching in \cite[Section~6.2]{ABS}. While it can be written as products of permutations of $R_2$ and $R_3$, there are exactly two cases in which Fano switching is not reducible to GM$_4$ and AH$_6$ switching. These are exactly the irreducible cases mentioned in \cite[Theorem~25]{ABS}.
\end{enumerate}
\end{example}

Strictly speaking, for the above design to yield exactly the same method as Fano switching, we also need the empty block, the block containing every point, and the blocks that are the complements of the lines. Remark~\ref{remark:main}(1) assures us that these blocks can always be added. Therefore, we do not include them in the incidence matrix.

\begin{remark}\label{remark:nonsimple}
    The above strategy also works for \emph{non-simple} designs (designs with repeated blocks), but we have to be more careful. The \texttt{AutomorphismGroup} of a design in GAP \cite{GAP,GAPdesign} is only defined for simple designs. Therefore, we could consider repeated blocks as the same block with a certain multiplicity that must be preserved under every automorphism. When calculating the matrix \(R\) from Theorem~\ref{thm:main}, we could modify the incidence matrix such that on position \((p,B)\), it has a $\sqrt{m}$ if the point \(p\) lies in the block \(B\), where \(m\) is the multiplicity of \(B\), and \(0\) otherwise. In this way, the expression for $R$ yields the same result as if we would have \(m\) columns that correspond to the same block \(B\).

    Despite the possibility, we restrict our classification to simple designs since there seems to be no useful upper bound on the number of blocks or the replication number \(r\) that a non-simple \((r,\lambda)\)-design on a fixed number of points can have. For example, the design that corresponds to GM$_{6+4}$~switching in Example~\ref{ex:GM64}, has \(816\) blocks in total and \(408\) blocks through each point. Therefore, a computer enumeration of non-simple designs seems infeasible.
\end{remark}

\subsubsection{Classification of switching methods from simple designs on at most 6 points}\label{sec:classification_simple}

In this section, we list all non-trivial switching methods that can be obtained from a simple design on \(v\) points, \(v\leq6\) (larger values of \(v\) become computationally difficult). That is, we allow multiple block sizes, but no repeated blocks.
For this classification, we make use of the \texttt{design} package in GAP \cite{GAPdesign}. To make the computations easier, we assume that there is no empty block or a block containing all points, since they do not change the \(R\) matrix from Theorem~\ref{thm:main}, see also Remark~\ref{remark:main}(i). We also assume that there is no block of size \(1\), because if there exists a block \(B=\{p\}\) of size one, then, up to a permutation of the rows and columns, the matrix \(R\) of Theorem~\ref{thm:main} is equal to \[R = \begin{pmatrix}1&0\\0&R'\end{pmatrix},\] where \(R'\) comes from the design where \(p\) is removed (note that deleting a point does not change \(r\) or \(\lambda\)). In the enumeration, this design would have already occurred. Similarly, we assume that there is no block of size \(v-1\).

Using the GAP code that is available on \url{https://github.com/robinsimoens/design-switching}, we conclude that there are exactly four switching methods that arise from a simple design on at most 6 points, see Table~\ref{tab:classification_simple}. They are listed as examples below.

\begin{table}[ht]
    \centering
    \begin{tabular}{c|c|l|l}
        $v$ & \# methods & Method name & Reference \\
        \hline\hline
        4 & 1 & GM$_4$ switching & Example~\ref{ex:GM4}\\
        \hline
        5 & 0 & \\
        \hline
        6 & 3 & GM$_6$ switching & Example~\ref{ex:GM6}\\
        & &WQH$_6$ switching & Example~\ref{ex:WQH6}\\
        & &AH$_6$ switching & Example~\ref{ex:AH6}
    \end{tabular}
    \caption{All switching methods from simple designs on \(v\) points, \(v\leq6\).}
    \label{tab:classification_simple}
\end{table}

\begin{example}[GM\(_4\)-switching]\label{ex:GM4} The affine plane AG$(2,2)$ is an \((r=3,\lambda=1)\)-design on four points with incidence matrix
$$\begin{pmatrix}
1&1&1&0&0&0\\
1&0&0&1&1&0\\
0&1&0&1&0&1\\
0&0&1&0&1&1\\
\end{pmatrix}.$$
It has six lines that can be permuted in 48 different ways that preserve parallelism.
The collineation group has 24 elements (all permutations of the four points), so there are two cosets. One of them is $H$ itself, corresponding to a trivial operation. A possible representative for the other one is a permutation that swaps every two lines of a parallel class. With the current ordering of the columns, it is the permutation $(1~6)(2~5)(3~4)$. This is exactly the GM$_4$ switching operation.
\end{example}

\begin{example}[GM\(_6\)-switching]\label{ex:GM6} The set of all 3-subsets of a 6-set form an \((r=10,\lambda=4)\)-design on six points with incidence matrix
$$\begin{pmatrix}
1&1&1&1&1&1&1&1&1&1&0&0&0&0&0&0&0&0&0&0\\
1&1&1&1&0&0&0&0&0&0&1&1&1&1&1&1&0&0&0&0\\
1&0&0&0&1&1&1&0&0&0&1&1&1&0&0&0&1&1&1&0\\
0&1&0&0&1&0&0&1&1&0&1&0&0&1&1&0&1&1&0&1\\
0&0&1&0&0&1&0&1&0&1&0&1&0&1&0&1&1&0&1&1\\
0&0&0&1&0&0&1&0&1&1&0&0&1&0&1&1&0&1&1&1
\end{pmatrix}.$$
There are two double cosets. The non-trivial one can be represented by the permutation \(B_i\mapsto B_{21-i}\) 
which is the same as for GM$_6$ switching.
\end{example}

\begin{example}[WQH\(_6\)-switching]\label{ex:WQH6} The \((r=4,\lambda=1)\)-design with the $6\times11$ incidence matrix
$$\begin{pmatrix}
1&1&1&1&0&0&0&0&0&0&0\\
1&0&0&0&1&1&1&0&0&0&0\\
1&0&0&0&0&0&0&1&1&1&0\\
0&1&0&0&1&0&0&1&0&0&1\\
0&0&1&0&0&1&0&0&1&0&1\\
0&0&0&1&0&0&1&0&0&1&1
\end{pmatrix}$$
together with the transposition \((1~11)\), gives rise to WQH-switching on 6 vertices.
\end{example}

\begin{example}[AH\(_6\)-switching]\label{ex:AH6} The \((r=7,\lambda=3)\)-design on six points with incidence matrix
$$\begin{pmatrix}
1&1&1&1&1&1&1&0&0&0&0&0&0&0\\
1&1&1&0&0&0&0&1&1&1&1&0&0&0\\
1&0&0&1&1&0&0&1&1&0&0&1&1&0\\
1&0&0&0&0&1&1&0&0&1&1&1&1&0\\
0&1&0&1&0&1&0&1&0&1&0&1&0&1\\
0&1&0&0&1&0&1&0&1&0&1&1&0&1
\end{pmatrix}$$
together with the permutation $(5~6~8)(7~10~9)$, yields the matrix $R_6$ that is used for AH$_6$ switching, see Section~\ref{sec:AH}.
\end{example}

\subsubsection{Classification of switching methods from simple designs on at most 8 points with constant block size}\label{sec:classification_simple_constant}

We continue the classification, but now restrict to switching methods from designs on \(v\) points, \(v\leq8\), with \emph{constant} block size \(k\). Again, larger values of \(v\) become computationally difficult. To avoid trivial cases, we assume that \(k\geq2\). We also assume that \(k\leq\frac{v}2\) because replacing a design by its complementary design does not change the \(R\) matrix from Corollary~\ref{cor:samedesign}. GAP code for the classification can be found on \url{https://github.com/robinsimoens/design-switching} and leads to the thirteen methods listed in Table~\ref{tab:classification_simple_constant}.

\begin{table}[ht]
    \centering
    \begin{tabular}{c|c|l|l|c}
        $v$ & \# methods & Method name & Reference & $k$\\
        \hline\hline
        4 & 1 & GM$_4$ switching & Example~\ref{ex:GM4} & $2$\\
        \hline
        5 & 0 & & \\
        \hline
        6 & 1 & GM$_6$ switching & Example~\ref{ex:GM6} & $3$\\
        \hline
        7 & 1 & Fano switching & Example~\ref{ex:fano} & $3$\\
        \hline
        8 & 10 & AG$(3,2)$-switching & Example~\ref{ex:AG32} & $4$
    \end{tabular}
    \caption{All switching methods from simple designs on \(v\) points, \(v\leq8\), with constant block size \(k\leq\frac{v}2\).}
    \label{tab:classification_simple_constant}
\end{table}

There are ten different switching methods with a switching set of size eight. Surprisingly, they can all be derived from the affine design AG$(3,2)$. For this reason, we bundle them under the name AG$(3,2)$-switching. They are listed in the following example.

\begin{example}[AG\((3,2)\)-switching]\label{ex:AG32} The eight points and fourteen planes of the affine space AG\((3,2)\) form an \((r=7,\lambda=3)\)-design with incidence matrix
$$N=\begin{pmatrix}
1&1&1&1&1&1&1&0&0&0&0&0&0&0\\
1&1&1&0&0&0&0&1&1&1&1&0&0&0\\
1&0&0&1&1&0&0&1&1&0&0&1&1&0\\
1&0&0&0&0&1&1&0&0&1&1&1&1&0\\
0&1&0&1&0&1&0&1&0&1&0&1&0&1\\
0&1&0&0&1&0&1&0&1&0&1&1&0&1\\
0&0&1&1&0&0&1&0&1&1&0&0&1&1\\
0&0&1&0&1&1&0&1&0&0&1&0&1&1
\end{pmatrix}.$$
There are fourteen double cosets with representatives:

\begin{align*}
    &\pi_1 = \mathrm{id}, & \pi_8 &= (3~12)(5~10)(6~8)(7~9),\\
    &\pi_2 = (6~7)(8~9), & \pi_9 &= (3~12)(5~9~7~10~6~8),\\
    &\pi_3 = (5~6~7)(8~10~9), & \pi_{10} &= (3~12)(5~10)(6~8~9~7),\\
    &\pi_4 = (3~4)(5~6~7)(8~10~9)(11~12), & \pi_{11} &= (6~7~9~8),\\
    &\pi_5 = (3~12)(5~10)(6~9)(7~8), & \pi_{12} &= (3~12)(5~9~7)(6~8~10),\\
    &\pi_6 = (3~12)(5~10)(6~9), & \pi_{13} &= (5~6~7~10~9~8),\\
    &\pi_7 = (7~8), & \pi_{14} &= (3~11~12~4)(5~9~7)(6~8~10).\\
\end{align*}

Defining $R_i=\frac14\left(N(N^{\pi_i})^T-3J_8\right)$, following the notation of Corollary~\ref{cor:samedesign}, we obtain the fourteen matrices in Table~\ref{tab:AG32}.

\begin{figure}[p]
    \centering
{\footnotesize
\begin{align*}
 R_1 &= I_8
 & R_8 &= \frac12 
 \left(\arraycolsep=1.5pt\def\arraystretch{.9}\begin{array}{rrrrrrrr}
  0&1&1&0&1&0&0&-1\\
  1&0&0&1&0&1&-1&0\\
  1&0&0&1&0&-1&1&0\\
  0&1&1&0&-1&0&0&1\\
  1&0&0&-1&1&0&0&1\\
  0&1&-1&0&0&1&1&0\\
  0&-1&1&0&0&1&1&0\\
  -1&0&0&1&1&0&0&1
 \end{array}\right)\\
 R_2 &= \frac12 
 \left(\arraycolsep=1.5pt\def\arraystretch{.9}\begin{array}{rrrrr}
  2I_4 & \0 & \0 & \0 & \0 \\
  \0& 1 & 1 & 1 & -1 \\
  \0& 1 & 1 & -1 & 1 \\
  \0& 1 & -1 & 1 & 1 \\
  \0& -1 & 1 & 1 & 1
 \end{array}\right)
 & R_9 &= \frac12 
 \left(\arraycolsep=1.5pt\def\arraystretch{.9}\begin{array}{rrrrrrrr}
  0&1&1&0&1&0&0&-1\\
  1&0&0&1&0&1&-1&0\\
  1&0&1&0&-1&0&1&0\\
  0&1&0&1&0&-1&0&1\\
  1&0&0&-1&1&0&0&1\\
  0&1&-1&0&0&1&1&0\\
  0&-1&0&1&1&0&1&0\\
  -1&0&1&0&0&1&0&1
 \end{array}\right)\\
  R_3 &= \frac12 
 \left(\arraycolsep=1.5pt\def\arraystretch{.9}\begin{array}{rrrrrrr}
  2I_2& \0& \0& \0&  \0&  \0&  \0 \\
  \0&   1&   1&   1&  -1&  0&   0 \\
  \0&   1&   1&  -1&   1&   0&   0 \\
  \0&   0&   0&   1&   1&   1&  -1 \\
  \0&   0&   0&   1&   1&  -1&   1 \\
  \0&   1&  -1&   0&   0&   1&   1 \\
  \0&  -1&   1&   0&   0&   1&   1
 \end{array}\right)
 & R_{10} &= \frac14 
 \left(\arraycolsep=1.5pt\def\arraystretch{.9}\begin{array}{rrrrrrrr}
1&1&1&1&1&1&1&-3\\
1&1&1&1&1&1&-3&1\\
1&1&1&1&1&-3&1&1\\
1&1&1&1&-3&1&1&1\\
3&-1&-1&-1&1&1&1&1\\
-1&3&-1&-1&1&1&1&1\\
-1&-1&3&-1&1&1&1&1\\
-1&-1&-1&3&1&1&1&1
\end{array}\right) \\
  R_4 &= \frac12 
 \left(\arraycolsep=1.5pt\def\arraystretch{.9}\begin{array}{rrrrrrrr}
  2 &  \0& \0&\0&\0&\0&\0&\0\\
  \0&  1&   1&   0&   1&   0&   0&  -1 \\
  \0&  1&   0&   1&   0&   -1&  0&   1 \\
  \0&  0&   1&   1&  -1&   1&   0&   0 \\
  \0&  1&  -1&   0&   0&   1&   1&   0 \\
  \0&  0&   0&   0&   1&   1&  -1&   1 \\
  \0&  0&   1&  -1&   0&   0&   1&   1 \\
  \0& -1&   0&   1&   1&   0&   1&   0
 \end{array}\right)
 & R_{11} &= \frac14 
 \left(\arraycolsep=1.5pt\def\arraystretch{.9}\begin{array}{rrrrrrrr}
3&1&1&-1&1&-1&-1&1\\
1&3&-1&1&-1&1&1&-1\\
1&-1&3&1&-1&1&1&-1\\
-1&1&1&3&1&-1&-1&1\\
-1&1&1&-1&3&1&1&-1\\
1&-1&-1&1&1&3&-1&1\\
1&-1&-1&1&1&-1&3&1\\
-1&1&1&-1&-1&1&1&3
\end{array}\right)\\
  R_5 &= \frac12 
 \left(\arraycolsep=1.5pt\def\arraystretch{.9}\begin{array}{rrrrrrrr}
  0&1&1&0&1&0&0&-1\\
  1&0&0&1&0&1&-1&0\\
  1&0&0&1&0&-1&1&0\\
  0&1&1&0&-1&0&0&1\\
  1&0&0&-1&0&1&1&0\\
  0&1&-1&0&1&0&0&1\\
  0&-1&1&0&1&0&0&1\\
  -1&0&0&1&0&1&1&0
 \end{array}\right)
 & R_{12} &= \frac14 
 \left(\arraycolsep=1.5pt\def\arraystretch{.9}\begin{array}{rrrrrrrr}
1&1&1&1&1&1&1&-3\\
1&1&1&1&1&1&-3&1\\
3&-1&1&1&-1&-1&1&1\\
-1&3&1&1&-1&-1&1&1\\
1&1&-1&-1&3&-1&1&1\\
1&1&-1&-1&-1&3&1&1\\
-1&-1&3&-1&1&1&1&1\\
-1&-1&-1&3&1&1&1&1
\end{array}\right)\\
  R_6 &= \frac14 
 \left(\arraycolsep=1.5pt\def\arraystretch{.9}\begin{array}{rrrrrrrr}
  1&1&1&1&1&1&1&-3\\
  1&1&1&1&1&1&-3&1\\
  1&1&1&1&1&-3&1&1\\
  1&1&1&1&-3&1&1&1\\
  1&1&1&-3&1&1&1&1\\
  1&1&-3&1&1&1&1&1\\
  1&-3&1&1&1&1&1&1\\
  -3&1&1&1&1&1&1&1
 \end{array}\right)
 & R_{13} &= \frac14 
 \left(\arraycolsep=1.5pt\def\arraystretch{.9}\begin{array}{rrrrrrrr}
3&1&1&-1&1&-1&-1&1\\
1&3&-1&1&-1&1&1&-1\\
-1&1&3&1&1&-1&1&-1\\
1&-1&1&3&-1&1&-1&1\\
1&-1&1&-1&1&3&1&-1\\
-1&1&-1&1&3&1&-1&1\\
1&-1&-1&1&1&-1&3&1\\
-1&1&1&-1&-1&1&1&3
\end{array}\right)\\
  R_7 &= \frac14 
 \left(\arraycolsep=1.5pt\def\arraystretch{.9}\begin{array}{rrrrrrrr}
  3&1&1&-1&1&-1&-1&1\\
  1&3&-1&1&-1&1&1&-1\\
  1&-1&3&1&-1&1&1&-1\\
  -1&1&1&3&1&-1&-1&1\\
  1&-1&-1&1&3&1&1&-1\\
  -1&1&1&-1&1&3&-1&1\\
  -1&1&1&-1&1&-1&3&1\\
  1&-1&-1&1&-1&1&1&3
 \end{array}\right)
 & R_{14} &= \frac14 
 \left(\arraycolsep=1.5pt\def\arraystretch{.9}\begin{array}{rrrrrrrr}
1&1&1&1&1&1&1&-3\\
3&1&1&-1&1&-1&-1&1\\
1&-1&1&3&-1&1&-1&1\\
-1&3&1&1&-1&-1&1&1\\
-1&1&-1&1&3&1&-1&1\\
1&1&-1&-1&-1&3&1&1\\
-1&-1&3&-1&1&1&1&1\\
1&-1&-1&1&1&-1&3&1
\end{array}\right)
\end{align*}}
\caption{Regular orthogonal matrices obtained from the design of points and planes in AG$(3,2)$.}
    \label{tab:AG32}
\end{figure}

The methods obtained from \(\pi_1,\dots,\pi_9\) are, respectively:
\begin{enumerate}
    \item Trivial.
    \item GM$_4$ switching.
    \item AH$_6$ switching.
    \item Fano switching.
In fact, the first four cases are exactly the four methods obtained from the Fano plane (Example~\ref{ex:fano}). This is no coincidence, because the Fano plane is a derived design from the point-plane design of AG$(3,2)$.
    \item GM$_{4+4}$ switching.
    \item GM$_8$ switching.
    \item WQH$_8$ switching.
    \item AH$_8$ switching.
    \item Cube switching.
\end{enumerate}
The methods from \(\pi_{10},\dots,\pi_{14}\) are new. We can prove that they are all reducible to GM$_4$, AH$_6$, GM$_8$ and WQH$_8$ switching using the SageMath \cite{sagemath} code that is available on \url{https://github.com/robinsimoens/design-switching}.

It was shown in \cite{ABS} that AH$_8$ and cube switching are reducible to GM$_4$, AH$_6$ and Fano switching. So all the matrices on the right in Figure~\ref{tab:AG32} lead to reducible methods.
\end{example}

Note that, although AH$_6$ switching appears in Example~\ref{ex:AG32}, we do not consider it to be a method from a design with constant block size, because the size of the switching set (\(6\)) is smaller than the number of points of the design (\(8\)). We only view it as a method coming from a design on six points, such as the design in Example~\ref{ex:AH6}.

\subsection{New switching methods from small designs}\label{sec:othernew}

In this section, we provide a bunch of new switching methods for constructing \(\mathbb{R}\)-cospectral graphs that do not belong to the classifications of Section~\ref{sec:classification_simple} and Section~\ref{sec:classification_simple_constant} because they come from larger (possibly non-simple) designs. We do not aim for any classification here.

\begin{example}[New switching method on seven vertices]\label{ex:new7} There is an $(r=15,\lambda=7)$-design with the following $7\times30$ incidence matrix:
$$\setlength{\arraycolsep}{2.9pt}\begin{pmatrix}
1&1&1&1&1&1&1&1&1&1&1&1&1&1&1&0&0&0&0&0&0&0&0&0&0&0&0&0&0&0\\
1&1&1&1&1&1&1&0&0&0&0&0&0&0&0&1&1&1&1&1&1&1&1&0&0&0&0&0&0&0\\
1&1&1&0&0&0&0&1&1&1&1&0&0&0&0&1&1&1&1&0&0&0&0&1&1&1&1&0&0&0\\
1&1&1&0&0&0&0&0&0&0&0&1&1&1&1&0&0&0&0&1&1&1&1&1&1&1&1&0&0&0\\
1&0&0&1&1&0&0&1&1&0&0&1&1&0&0&1&1&0&0&1&1&0&0&1&1&0&0&1&1&0\\
0&1&0&1&0&1&0&1&0&1&0&1&0&1&0&1&0&1&0&1&0&1&0&1&0&1&0&1&0&1\\
0&0&1&0&1&1&0&0&1&1&0&0&1&1&0&1&0&0&1&1&0&0&1&1&0&0&1&0&1&1
\end{pmatrix}.$$
Together with the permutation \[(1~24)(2~20)(3~16)(4~14)(5~10)(7~30)(8~13)(11~29)(15~28)(17~27)(18~23)(21~26)\] it leads to the regular orthogonal matrix
$$\frac14\left(\begin{array}{r|rrr|rrr}
    1 & -1 & -1 & -1 & 2 & 2 & 2\\
    \hline
    -1 & 1 & 1 & 1 &-2 & 2 & 2\\
    -1 & 1 & 1 & 1 & 2 & -2 & 2\\
    -1 & 1 & 1 & 1 & 2 & 2 & -2\\
    \hline
    2 & -2 & 2 & 2 & 0 & 0 & 0\\
    2 & 2 & -2 & 2 & 0 & 0 & 0\\
    2 & 2 & 2 & -2 & 0 & 0 & 0
\end{array}\right).$$
The corresponding switching method is new. However, it does not provide new examples of cospectral graphs: one can check by computer (using the code that is available on \url{https://github.com/robinsimoens/design-switching}) that it is always reducible to GM$_4$ switching.
\end{example}

\begin{example}[New switching method on eight vertices]\label{ex:new8} Consider the \emph{non-simple} $(r=23,\lambda=11)$-design with the $8\times30$ incidence matrix
$$\setlength{\arraycolsep}{2.9pt}\begin{pmatrix}
1&1&1&2&2&1&1&1&1&2&2&2&2&2&2&0&0&0&0&0&0&0&0&0&0&0&0&0&0&0\\
1&1&1&2&2&1&1&1&1&0&0&0&0&0&0&2&2&2&2&2&2&0&0&0&0&0&0&0&0&0\\
1&1&1&2&0&0&0&0&0&2&2&2&0&0&0&2&2&2&0&0&0&1&1&1&1&2&0&0&0&0\\
1&1&1&0&2&0&0&0&0&2&0&0&2&2&0&2&0&0&2&2&0&1&1&1&1&0&2&0&0&0\\
1&0&0&2&0&1&1&0&0&0&2&0&2&2&0&2&0&0&2&0&2&1&1&0&0&2&0&1&1&0\\
1&0&0&0&2&1&1&0&0&2&2&0&0&0&2&0&2&2&2&0&0&1&1&0&0&0&2&1&1&0\\
0&1&0&2&0&1&0&1&0&2&0&0&2&0&2&0&2&0&2&2&0&1&0&1&0&2&0&1&0&1\\
0&1&0&0&2&1&0&1&0&0&2&2&2&0&0&2&2&0&0&0&2&1&0&1&0&0&2&1&0&1\\
\end{pmatrix}.$$
The entries ``$2$'' indicate that a blocks has multiplicity $2$.
Together with the permutation $(4~5)(10~16)(15~21)(26~27)$
, it gives rise to the orthogonal matrix
$$\frac13\left(\begin{array}{rrrr|rrrr}
2&1&0&0&1&-1&-1&1\\
1&2&0&0&-1&1&1&-1\\
0&0&2&1&-1&1&-1&1\\
0&0&1&2&1&-1&1&-1\\
\hline
1&-1&-1&1&1&2&0&0\\
-1&1&1&-1&2&1&0&0\\
-1&1&-1&1&0&0&1&2\\
1&-1&1&-1&0&0&2&1
\end{array}\right).$$
Up to symmetry, there are ten different $A_C$ matrices for which the method is not reducible to other smaller known methods, see Figure~\ref{tab:ACmatsnew8}. For all 62 other matrices, it is reducible to WQH$_6$ switching.
\end{example}

\begin{example}[New switching method on eight vertices]\label{ex:new8b}
The regular orthogonal matrix
\[\frac15\mathrm{circulant}(3,1,2,-1,-2,1,2,-1).\]
gives rise to a switching method that is not reducible to any of the previous methods. There are 98 different $A_C$ matrices up to symmetry. We do not list them.
\end{example}

\newcommand{\colsep}{.6pt}
\newcommand{\stret}{.5}
\begin{figure}[ht]
$$
\left(\arraycolsep=\colsep\def\arraystretch{\stret}\begin{array}{cccccccc}
0&0&0&0&0&0&0&1\\0&0&0&0&0&0&1&0\\0&0&0&1&1&0&0&0\\0&0&1&0&0&1&0&0\\0&0&1&0&0&0&0&0\\0&0&0&1&0&0&0&0\\0&1&0&0&0&0&0&1\\1&0&0&0&0&0&1&0
\end{array}\right)
\left(\arraycolsep=\colsep\def\arraystretch{\stret}\begin{array}{cccccccc}
0&0&0&0&0&0&0&1\\0&0&0&0&0&0&1&0\\0&0&0&1&1&0&0&0\\0&0&1&0&0&1&0&0\\0&0&1&0&0&0&1&1\\0&0&0&1&0&0&1&1\\0&1&0&0&1&1&0&1\\1&0&0&0&1&1&1&0
\end{array}\right)
\left(\arraycolsep=\colsep\def\arraystretch{\stret}\begin{array}{cccccccc}
0&0&0&0&0&0&0&1\\0&0&0&0&0&0&1&0\\0&0&0&1&1&0&1&1\\0&0&1&0&0&1&1&1\\0&0&1&0&0&0&0&0\\0&0&0&1&0&0&0&0\\0&1&1&1&0&0&0&1\\1&0&1&1&0&0&1&0\\
\end{array}\right)
\left(\arraycolsep=\colsep\def\arraystretch{\stret}\begin{array}{cccccccc}
0&0&0&0&0&0&0&1\\0&0&0&0&0&0&1&0\\0&0&0&1&1&0&1&1\\0&0&1&0&0&1&1&1\\0&0&1&0&0&0&1&1\\0&0&0&1&0&0&1&1\\0&1&1&1&1&1&0&1\\1&0&1&1&1&1&1&0
\end{array}\right)
\left(\arraycolsep=\colsep\def\arraystretch{\stret}\begin{array}{cccccccc}
0&0&0&0&0&1&0&1\\0&0&0&0&1&0&1&0\\0&0&0&0&0&1&1&0\\0&0&0&0&1&0&0&1\\0&1&0&1&0&0&0&0\\1&0&1&0&0&0&0&0\\0&1&1&0&0&0&0&0\\1&0&0&1&0&0&0&0
\end{array}\right)$$
$$
\left(\arraycolsep=\colsep\def\arraystretch{\stret}\begin{array}{cccccccc}
0&0&0&0&0&1&0&1\\0&0&0&0&1&0&1&0\\0&0&0&0&0&1&1&0\\0&0&0&0&1&0&0&1\\0&1&0&1&0&0&1&1\\1&0&1&0&0&0&1&1\\0&1&1&0&1&1&0&0\\1&0&0&1&1&1&0&0
\end{array}\right)
\left(\arraycolsep=\colsep\def\arraystretch{\stret}\begin{array}{cccccccc}
0&0&0&0&1&0&1&1\\0&0&0&0&0&1&1&1\\0&0&0&1&0&0&1&0\\0&0&1&0&0&0&0&1\\1&0&0&0&0&1&0&0\\0&1&0&0&1&0&0&0\\1&1&1&0&0&0&0&0\\1&1&0&1&0&0&0&0
\end{array}\right)
\left(\arraycolsep=\colsep\def\arraystretch{\stret}\begin{array}{cccccccc}
0&0&0&0&1&0&1&1\\0&0&0&0&0&1&1&1\\0&0&0&1&0&0&1&0\\0&0&1&0&0&0&0&1\\1&0&0&0&0&1&1&1\\0&1&0&0&1&0&1&1\\1&1&1&0&1&1&0&0\\1&1&0&1&1&1&0&0
\end{array}\right)
\left(\arraycolsep=\colsep\def\arraystretch{\stret}\begin{array}{cccccccc}
0&0&0&0&1&0&1&1\\0&0&0&0&0&1&1&1\\0&0&0&1&1&1&1&0\\0&0&1&0&1&1&0&1\\1&0&1&1&0&1&0&0\\0&1&1&1&1&0&0&0\\1&1&1&0&0&0&0&0\\1&1&0&1&0&0&0&0
\end{array}\right)
\left(\arraycolsep=\colsep\def\arraystretch{\stret}\begin{array}{cccccccc}
0&0&1&1&0&1&0&1\\0&0&1&1&1&0&1&0\\1&1&0&0&0&1&1&0\\1&1&0&0&1&0&0&1\\0&1&0&1&0&0&1&1\\1&0&1&0&0&0&1&1\\0&1&1&0&1&1&0&0\\1&0&0&1&1&1&0&0
\end{array}\right)$$
\caption{Adjacency matrices for which the switching method of Example~\ref{ex:new8} is irreducible, up to conjugation and complementation.}
    \label{tab:ACmatsnew8}
\end{figure}

\begin{example}[AG$(2,3)$-switching]\label{ex:AG23}

The affine plane of order 3 is an $(r=4, \lambda=1)$-design on nine points. It has twelve lines that can be permuted in essentially five (the number of double cosets) different ways that preserve parallelism:
\begin{enumerate}
    \item One permutation is trivial.
    \item A transposition of two parallel lines corresponds to WQH$_6$-switching (see Remark~\ref{rem:planesreducible}).
    \item Another one gives rise to the method from Example~\ref{ex:new8}.
    \item A cyclic permutation of the three lines in a parallel class leads to the regular orthogonal matrix \[\frac13\mathrm{circulant}(2,1,0,-1,1,0,-1,1,0)\] which, up to permutation, is the same as \(Q(3,3,1)\) in \cite[Section~4.2]{ABScounting}.
    \item Finally, there is a permutation that yields the (new) matrix \[\frac13\mathrm{circulant}(\mathrm{circulant}(2,1,0),\mathrm{circulant}(0,-1,1),\mathrm{circulant}(0,-1,1)).\]
\end{enumerate}
\end{example}

\begin{example}[Paley biplane switching]\label{ex:paley}
The Paley biplane is the unique $2$-\allowbreak$(11,5,2)$-design. Following the same method as in the prior subsection, we obtain $126$ double cosets of $H\backslash G/H$, where \(G\) is the group of permutations of blocks that preserve their intersection sizes, and \(H\) is the automorphism group of the design, acting on its blocks.
More precisely, we have a distribution of $0^1, 6^1, 8^3, 9^8, 10^{36}, 11^{77}$, where we mean by $a^b$ that
\[
 R = \begin{pmatrix}
      I_{n-a} & 0 \\
      0 & R'
     \end{pmatrix}
\]
for $b$ double cosets.
Due to their large number, we did not investigate them in detail.
\end{example}

\begin{example}[PG$(2,3)$-switching]\label{ex:PG23} The projective plane of order 3 is an $(r=4,\lambda=1)$-design on thirteen points. There are 252 double cosets of $H\backslash G/H$, where \(G\cong \Sym(13)\) and \(H\cong\mathrm{P\Gamma L}(2,3)\) is the automorphism group of PG$(2,3)$, acting on its lines. 
Following the notation of Example~\ref{ex:paley}, these are distributed as $0^1, 6^1, 8^1, 9^3, 10^4, 11^8, \allowbreak 12^{47},\allowbreak 13^{187}$.
Again, there are too many to properly describe them all.

All obtained regular orthogonal matrices have level 3 (except the identity matrix). This suggests that classifying switching methods of level 3 is probably infeasible, unlike those of level 2, as was done in \cite{ABS}.
\end{example}

\section{Incidence structures from switching methods} \label{sec:notallswitchings}

We can also go the other way around and construct incidence structures from a given switching method. Let \(\cP\) be the set of vertices of the switching set. Let \(\cB_1\) be the set of possible neighbourhoods within \(\cP\) for a vertex outside the switching set, before switching. Let \(\cB_2\) be the set of possible neighbourhoods within \(\cP\) for such a vertex after switching.

Algebraically, if the switching method corresponds to a conjugation of the adjacency matrix with the orthogonal matrix \(Q(R)\), then the characteristic vectors of the blocks of \(\cB_1\) are exactly those 01-vectors \(\chi\) such that \(R^T\chi\) is again a 01-vector.
The characteristic vectors of the blocks of \(\cB_2\) are all their images \(R^T\chi\).
The incidence structures $(\cP,\cB_1)$ and $(\cP,\cB_2)$ are not necessarily \((r,\lambda)\)-designs. But if they are, then the switching method is exactly the method obtained through Theorem~\ref{thm:main}.

If the matrix \(R\) is regular, then every point lies in \(r=\frac12|\cB_1|\) blocks (consider the map \(\chi\mapsto\1-\chi\) where \(\1\) is the all-one vector). Still, the property that every two distinct points lie on the same number of blocks, is in general not true. Sometimes adding multiplicities to certain blocks can fix this. In general it can not. We give an example of a switching method for which \(R\) is regular and indecomposable, but the incidence structure is not an \((r,\lambda)\)-design and cannot be made such a design by adding multiplicities.

\begin{prop}\label{thm:newnodesign}
    Let \(\Gamma\) be a graph with a subgraph \(C\) with vertex set \(V(C)=\{1,\dots,6\}\) such that:
    \begin{enumerate}[(i)]
        \item The adjacency matrix of \(C\) is either edgeless, or the union of two triangles on \(\{1,4,5\}\) and \(\{2,3,6\}\), or the union of the edge \(\{1,2\}\) and a complete graph on \(\{3,4,5,6\}\).
        \item For every vertex \(x\notin C\), \[\deg(x,\{4,5,6\})\equiv2\cdot\left(\deg(x,\{1\})-\deg(x,\{2,3\})\right)\pmod5.\]
    \end{enumerate}
    Make every vertex \(x\notin C\) that has exactly three neighbours in \(C\), adjacent to the other three instead. The resulting graph is \(\mathbb{R}\)-cospectral with \(\Gamma\).
\end{prop}
\begin{proof}
    This is the combinatorial interpretation of conjugating the adjacency matrix with the regular orthogonal matrix
$$R=\frac15\left(\begin{array}{r|rr|rrr}
    2 & 3 & 3 & -1 & -1 & -1\\
    \hline
    3 & 2 & -3 & 1 & 1 & 1\\
    3 & -3 & 2 & 1 & 1 & 1\\
    \hline
    -1 & 1 & 1 & -2 & 3 & 3\\
    -1 & 1 & 1 & 3 & -2 & 3\\
    -1 & 1 & 1 & 3 & 3 & -2
\end{array}\right).$$
One checks (by computer) that:
\begin{enumerate}
    \item The possible matrices \(A_C\) for which \(R^TA_CR\) is an adjacency matrix again, are the adjacency matrices of the graphs of condition (i), up to symmetry (complementation and conjugation). For all three of them, \(R^TA_CR=A_C\).
    \item The 01-vectors \(v\) for which \(R^Tv\) is a 01-vector again, give rise to exactly the neighbourhoods of condition (ii).
    The 01-vectors \(v\) such that \(R^Tv\neq v\), are exactly those with three ones. \qedhere
\end{enumerate}
\end{proof}

One also checks that the above method is not reducible to any of the previously known methods.
We show that it cannot be obtained by design switching:

\begin{prop}
The switching method from Proposition~\ref{thm:newnodesign} is not a special case of Theorem~\ref{thm:main}.
\end{prop}
\begin{proof}
    Suppose, by contradiction, that the method is coming from an $(r,\lambda)$-design. It has twelve blocks with certain multiplicities. They correspond to the possible neighbourhoods of a vertex \(x\notin C\), see condition (ii) of Proposition~\ref{thm:newnodesign}. If we denote \(m_B\) for the multiplicity of block \(B\), then we have, by looking at the number of blocks through the points \(2\) and \(3\), that \(\lambda=m_{\{2,3,4\}}+m_{\{2,3,5\}}+m_{\{2,3,6\}}+m_{\{1,\dots,6\}}\). Similarly, by looking at the pairs of points \(\{2,4\}\), \(\{2,5\}\), \(\{3,6\}\), \(\{4,5\}\) and \(\{4,6\}\), we get five more equations that have (positive) integer solutions, but this is not the case.
\end{proof}

\section{Applications}\label{sec:appl}

Here, we list several constructions of cospectral graphs using Theorem \ref{thm:main}.
We do not aim for completeness. For instance, we skip the countless constructions
in the tradition of Wallis \cite{Wallis} which already explicitly utilise designs.

\subsection{Trades in designs}

Theorem~\ref{thm:main} is closely linked to the concept of trades in designs, see \cite{Billington2003,KMT,Ostergard}. If the design used in Theorem~\ref{thm:main} is symmetric and our graph is the intersection graph of a partial linear space, then we arrive at the following result.\footnote{%
Trades in designs are clearly known. Our point here is that they yield cospectral graphs.}

\begin{prop}\label{prop:traddesigns}
    Consider a partial linear space $\cD=(\cP,\cL)$ with a projective plane \(\Pi_1=(\cP_1,\cL_1)\) as a subgeometry, that is, \(\cP_1\subseteq\cP\) and \(\cL_1\subseteq\cL\). If $\Pi_2=(\cP_1,\cL_2)$ is a projective plane on the same point set, then the block graphs of $(\cP,\cL)$ and $(\cP,(\cL \setminus \cL_1) \cup \cL_2)$ are $\mathbb{R}$-cospectral.
\end{prop}
\begin{proof}
    Let $\Pi_1^*$ and $\Pi_2^*$ denote the dual projective planes of $\Pi_1$ and $\Pi_2$ respectively. Since $\Pi_1$ and $\Pi_2$ have the same number of points, they have the same order. The induced subgraph on $\cL_1$ is complete. Choose $L\in\cL\setminus\cL_1$. If \(L\cap\cP_1=\emptyset\), then $L$ has no neighbours in \(\cL_1\). If \(L\cap\cP_1=\{p\}\) for some point \(p\), then the neighbourhood of $L$ in $\cL_1$ is the set of lines through \(p\), which corresponds to a block of $\Pi_1^*$, and similarly for the neighbourhood of $L$ in $\cL_2$. Thus we can apply Theorem \ref{thm:main} with the designs $\Pi_1^*$ and $\Pi_2^*$ to obtain the result.
\end{proof}

\subsection{Cospectral mates for \texorpdfstring{$q$}{q}-triangular graphs}

Let us discuss a specific case of Proposition \ref{prop:traddesigns}.
Let \(q\) be a prime power and \(n\geq4\). The \emph{q-triangular graph} $J_q(n, 2)$ has as vertices the lines of PG\((n-1,q)\) where two vertices are adjacent if they intersect.

Let $\cP$ and $\cB$ be the point set and line set of PG\((n-1,q)\) respectively. The graph $J_q(n, 2)$ is exactly the block graph of $(\cP,\cB)$. Every projective plane PG\((2,q)\subseteq\PG(n-1,q)\) forms a symmetric $2$-$(q^2+q+1, q+1, 1)$ subdesign. Proposition~\ref{prop:traddesigns} tells us that we can replace it by another $2$-$(q^2+q+1, q+1, 1)$-design, that is, a projective plane of order $q$,
and get an \(\mathbb{R}\)-cospectral graph. We can also apply a permutation that is not an automorphism. In fact, every permutation \(\pi\) of \(\cB\) that is not an automorphism, creates maximal cliques of size \(q+2\). This new graph cannot be isomorphic to the original one, (re)proving that

\begin{theorem}[\cite{IM}]
    $J_q(n, 2)$, \(n\geq4\), is not determined by its spectrum.
\end{theorem}

We use this example to give a general sketch
for how one can use Corollary \ref{cor:samedesign} to show that there are many
non-isomorphic graphs with a given spectrum whenever design switching is applicable.\footnote{%
The first author believes that the particular result here is well-known,
but we could not find a suitable reference.} This result is stated in Theorem~\ref{thm:many_Johnson}.

We do this in a similar way as in \cite[Section~7.2]{BIK}, but the design is different.
Fix a projective plane \(\alpha\) and consider the design \((\mathcal{P},\mathcal{B})\) where
\begin{align*}
    \mathcal{P}&=\{\text{lines of }\alpha\},\\
    \mathcal{B}&=\{\text{point-pencils of }\alpha\},
\end{align*}
where a \emph{point-pencil} is a set of all lines through a given point. This design allows us to apply Theorem~\ref{thm:main} with any other \((r=q+1,\lambda=1)\)-design. Indeed, each vertex of the graph that is not in \(\cP\) (a line not in \(\alpha\)), is adjacent to either a point-pencil of \(\alpha\) or no line of \(\alpha\). But in order to prove Theorem~\ref{thm:many_Johnson}, it suffices to use Corollary~\ref{cor:samedesign} where we use the same design with the blocks permuted. Given such a permutation \(\pi\), denote the new graph by \(\Gamma_\pi\). Since \(\pi\) permutes the point-pencils, we can (and will) look at it as a permutation of the points of \(\alpha\).

\begin{lemma}
    If \(\pi\notin P\Gamma L(\alpha)\), then there are at least three non-collinear lines \(\ell\) in \(\alpha\) such that \(\ell^\pi\) is not a line.
\end{lemma}
\begin{proof}
    There is at least one such line \(\ell\), since otherwise \(\pi\) would be a collineation. Pick a point \(p\in\ell\). If all other lines through \(p\) would be mapped to lines, then \(\ell^\pi\) would be contained in \(\alpha\setminus\bigcup_{\ell'\supseteq p}\ell'^\pi\), which is contained in a line, a contradiction. So there is at least one other such line. Repeating that argument once more with another point on \(\ell\) gives us a third such line.
\end{proof}

\begin{lemma}
    If \(\pi\notin P\Gamma L(\alpha)\), then $\cP$ is recognisable in \(\Gamma_\pi\).
\end{lemma}
\begin{proof}
    We prove that
    \begin{center}
        \emph{maximal cliques of size at most \(q+2\) contain exactly one vertex of \(\cP\)}
    \end{center}
    by contraposition.

    If a maximal clique is disjoint from \(\cP\), it is necessarily contained in a maximal clique of \(\Gamma\) with its vertices in \(\cP\) removed. Such a clique is either a point-pencil or its dual (all lines in a given plane). A point-pencil through a point \(p\in\alpha\) without its lines in \(\alpha\) is not maximal as we can add the point-pencil through \(p^\pi\) in \(\alpha\) (call this a shifted point-pencil). A point-pencil through a point \(p\notin\alpha\) has \(\frac{q^{n-1}-1}{q-1}\) lines, a plane-pencil has size \(q^2+q+1\) and has an intersection of at most one line with \(\cP\).

    If a maximal clique contains at least two vertices of \(\cP\), it is either \(\cP\) itself or a shifted point-pencil as before.

    Now we claim that \(\cP\) is characterised by the above property. Suppose that \(C\) is a maximal clique with the property. If \(\ell\) is a line in \(\alpha\) such that \(\ell^{\pi^{-1}}\) is not a line, and if \(p\) is a point not in \(\alpha\), then \[\{\ell\}\cup\{pq\colon\, q^\pi\in\ell\}\] is a maximal clique of size \(q+2\). Now, use the above lemma to find three non-collinear such lines. If \(C\) contains all three of them, then it is equal to \(\cP\). If not, we can use the property on the maximal cliques of size \(q+2\) and conclude that there is a line in \(C\) that goes through \(p\). But since \(p\notin\alpha\) can be chosen arbitrarily, the union of the lines of \(C\) must cover all points of PG\((n-1,q)\setminus\alpha\). So it must be a point-pencil, which this contradicts the ``exactly one'' part of the property.
\end{proof}

\begin{lemma}\label{lem:doublecoset}
    We have \(\Gamma_\pi\cong\Gamma_\tau\) if and only if \(\pi\) and \(\tau\) are in the same double coset of \(P\Gamma L(\alpha)\) (acting on the lines) in Sym\((\cP)\).
\end{lemma}
\begin{proof}
    Note that \(\pi\in P\Gamma L(\alpha)\Longleftrightarrow\Gamma_\pi\cong\Gamma\) because if \(\ell\) is a line such that \(\ell^{\pi^{-1}}\) is not a line, and if \(p\) is a point not in \(\alpha\), then \[\{\ell\}\cup\{pq\colon\, q^\pi\in\ell\}\] is a maximal clique of size \(q+2\). But maximal cliques in \(\Gamma\) have size at least \(q^2+2+1\).

    So suppose that \(\pi,\tau\notin P\Gamma L(\alpha)\) and let \(f:\Gamma_\pi\to\Gamma_\tau\) be an isomorphism. By the above lemma, \(\cP\) is recognisable in both graphs, so \(f(\cP)=\cP\).

    Now, \(f\) maps point-pencils in \(\alpha\) to point-pencils in \(\alpha\), since those are exactly the cliques of size \(q+1\) in \(\cP\) that can be extended to larger cliques in the full graph. So \(f|_\cP\in P\Gamma L(\alpha)\) (acting on the lines).

    The point-pencils in PG\((n-1,q)\) through a point \(p\notin\alpha\) are also mapped to such point-pencils again, since they are the only cliques of size \(\frac{q^{n-1}-1}{q-1}\) disjoint from \(\cP\). So \(f|_{V(\Gamma)\setminus\cP}\) preserves collinear lines and can be uniquely extended to a collinearity \(g\in P\Gamma L(n,q)\).

    We show that \(\pi f=g \tau\) as maps on the lines of \(\alpha\), thereby proving the claim. Pick a line \(\ell\in\cP\). Let \(p\) be an arbitrary point on \(\ell\) and let \(m\) be a line through \(p\). Then \(m\sim\ell^\pi\), so \(f(m)\sim f(\ell^\pi)\). In other words, \(f(\ell^\pi)\) goes through \((\alpha\cap f(m))^\tau=g(p)^\tau\). So \(f(\ell^\pi)=g(\ell)^\tau\).
\end{proof}

\begin{theorem}\label{thm:many_Johnson}
    There are at least $q!$ non-isomorphic graphs with the same spectrum as $J_q(n, 2)$, \(n\geq4\).
\end{theorem}
\begin{proof}
    There are at least \(q!\) double cosets of \(P\Gamma L(2,q)\) in Sym\((q^2+q+1)\) because \(|P\Gamma L(2,q)|<q^9\) and \(|\Sym(q^2+q+1)|=(q^2+q+1)!\). Applying Lemma~\ref{lem:doublecoset} gives the result.
\end{proof}

\subsection{Polar spaces}

The goal of this section is to fit several constructions from finite classical polar spaces into our framework.
In particular, we explain the construction in \cite[Section~7.1]{BIK} algebraically.
It also provides a new perspective on other publications such as \cite{Guo,I,IM}.

Let $q$ be a prime power.
There are five types of finite classical polar spaces: $O^+(2d, q)$, $O(2d+1, q)$, $O^-(2d+2)$,
$U(n, \sqrt{q})$, $Sp(2d, q)$, where $d$ and $n$ are positive integers. They are related to various strongly regular graphs, see \cite[Chapter~2 and Chapter~3]{BvM}. Each of the finite classical polar spaces can be defined by a suitable form over GF$(q)$:
\begin{itemize}
 \item For $O^+(2d, q)$, the quadratic form $Q(x) = x_1 x_2 + \cdots + x_{2d-1} x_{2d}$.
 \item For $O(2d+1, q)$, the quadratic form $Q(x) = x_1^2 + x_2x_3 + \cdots + x_{2d} x_{2d+1}$.
 \item For $O^-(2d+2, q)$, the quadratic form $Q(x) = x_1^2 + \alpha x_1 x_2 + \beta x_2^2 + \cdots + x_{2d+1} x_{2d+2}$,
        where $x^2 + \alpha x + \beta$ is irreducible over GF$(q)$.
 \item For $U(n, \sqrt{q})$, the Hermitian form $\sigma(x, y) = x_1 y_1^{\sqrt{q}} + \cdots + x_n y_n^{\sqrt{q}}$.
 \item For $Sp(2d, q)$, the symplectic form $\sigma(x, y) = x_1 y_2 - x_2 y_1 + \cdots + x_{2d-1} y_{2d} - x_{2d} y_{2d-1}$.
\end{itemize}
We say that a vector $x$ is \textit{singular} with respect to $Q$ if $Q(x) = 0$. A vector \(x\) is \textit{isotropic} if with respect to $\sigma$ if $\sigma(x, x) = 0$.
More generally, a subspace is \textit{totally singular} if all its vector are singular
and a subspace $S$ is \textit{totally isotropic} if $\sigma(x, y) = 0$ for all $x, y \in S$.
For simplicity, we use \textit{isotropic} as synonym for singular, totally singular,
and totally isotropic.
Unless $q$ is even, we can associate the finite classical polar space with
a polarity $\perp$ of $\mathrm{PG}(n-1, q)$, that is, $\perp$ is an incidence preserving
bijection of order two on the subspaces of $\mathrm{PG}(n-1, q)$
with $\dim(U^\perp) = \mathrm{codim}(U)$ for any subspace.
For a subspace $S$, the \textit{radical} of $S$ is the
set of vectors $x \in S$ such that $\< x, y \>$ is isotropic
for all $y \in S$. If $\perp$ exists, then $S \cap S^\perp$ is the radical of $S$.

We write $[k]_q = \frac{q^k-1}{q-1}$ for the number of points of PG$(k-1,q)$.

\subsubsection{Adjacent when orthogonal}\label{sec:orthogonal}

Let $\Gamma$ be a graph whose vertex set $\V$ is a subset of the points of PG\((n-1,q)\), where two distinct points $x$ and $y$ are adjacent if $x \in y^\perp$.

\begin{lemma}\label{lemma:polar_orth}
    Let $U$ be a subspace of PG\((n-1,q)\) of dimension $k-1$ and let \(\perp\) be a polarity of PG\((n-1,q)\). The restriction of the point-hyperplane design of $U$ to $\V \cap U$ or $\V \cap U \setminus U^\perp$ is an \((r=[k-1]_q,\lambda=[k-2]_q)\)-design that satisfies the (relaxed, see Remark~\ref{remark:main}(1)) condition (ii') of Theorem \ref{thm:main}.
\end{lemma}
\begin{proof}
    Within \(U\), every point is contained in \([k-1]_q\) hyperplanes and every line is contained in \([k-2]_q\) hyperplanes. Moreover, deleting points of the design does not change $r$ or $\lambda$.

    For a point $x\notin U$, the hyperplane $x^\perp$ intersects $U$ in either a hyperplane of $U$ or $U$ itself. Thus \(x\) is adjacent to a block of the point-hyperplane design of $U$ or all points of $U$. If \(x\in U^\perp\), then it is adjacent to all points of \(U\).
\end{proof}

There are the following natural choices for $\V$:
\begin{itemize}
 \item All isotropic points. Note that we can include $q$ even and quadratic forms if we require that $\< x, y\>$ is singular.
 \item All anisotropic points.
 \item For quadratic forms and $q$ odd, one isometry type of anisotropic points.
\end{itemize}
For applying Theorem \ref{thm:main}, we need to find a subspace $U$ such that for $C=\V\cap U$ or $C=\V\cap U\setminus U^\perp$ (as in Lemma \ref{lemma:polar_orth}) the matrix $R^T A_C R$ is an adjacency matrix. The easiest case is when $C$ is a clique or a coclique.

\begin{example}
    If $U$ is a totally isotropic subspace and $C = \V \cap U$, then \(C\) is a clique. Thus we obtain cospectral graphs. If $U$ is a plane or bigger, then we can hope for non-isomorphy.
    There are two families of strongly regular graphs related to this case. If $\V$ consists of all isotropic points, then this is the example in \cite[Section~7.1]{BIK}. If the polar space is a hyperbolic quadric and $\V$ consists of all isotropic points minus one maximal isotropic subspace, then it is the construction in \cite{OtherBIK}.
\end{example}

\begin{example}
 Suppose that $\V$ consists of anisotropic points only. Take a subspace $U$ such that \(U\cap U^\perp\) is a hyperplane of \(U\). If we put $C = \V \cap U \setminus U^\perp$, then $C$ is a coclique. Thus we obtain cospectral graphs. Again, if $U$ is a plane or bigger, then we can hope for non-isomorphy.

 If $\V$ consists of all anisotropic points and the polarity is Hermitian, then $\Gamma$ is the strongly regular graph $NU_n^\perp(q)$ in the notation of \cite{BvM}.
 If $\perp$ is quadratic and $\V$ consists of one isometry type, then $\Gamma$ is a strongly regular graph in some cases for $q \in \{3,5\}$.
 This generalizes several constructions in \cite{IM}.
\end{example}

\begin{example}
  Suppose that $\perp$ is symplectic or Hermitian and that $\V$ consists of all isotropic points. Take a subspace $U$ such that \(U\cap U^\perp\) has codimension 2 in \(U\). If $C =\V\cap U\setminus U^\perp$, then the induced subgraph on $C$ is a disjoint union of cliques of size $q^{\dim(U)}$. We have $q+1$ cliques if the polar space is symplectic and $\sqrt{q}+1$ cliques if it is Hermitian.
  For $n=5$ and $\dim(U)=2$, this is the construction in \cite{Guo}.

  Note that this construction is also feasible for quadrics, but here we only obtain $2$ cliques. The collineation group contains an element exchanging these two, so here no non-isomorphic graphs can be found.
\end{example}

\subsubsection{Adjacent when tangential}

Let $\Gamma$ be a graph whose vertex set \(\V\) is a subset of points of PG$(n-1,q)$, where two distinct points $x$ and $y$ are adjacent if the line $\<x,y\>$ spanned by them, is a tangent (that is, $\<x,y\>$ contains precisely one isotropic point). As before, if $G$ is the subgroup of $\mathrm{P\Gamma{}L}(n, q)$ that stabilises $\V$ set-wise, then $G$ induces a subgroup of $\mathrm{Aut}(\Gamma)$.

In the notation of \cite{BvM}, the strongly regular graph $NU_n(q)$ equals $\Gamma$ if $\V$ is the set of all anisotropic points and $\perp$ is Hermitian.
For $n \geq 5$, one can use WQH-switching to show that $NU_n(q)$ is not determined by its spectrum \cite{IPS}.

Below, we give a generalization of the method, in a similar fashion as in the Section~\ref{sec:orthogonal}.

\begin{example}
    Suppose that $\V$ consists of all anisotropic points. Take a subspace $U$ such that \(R=U\cap U^\perp\) is a hyperplane in \(U\). If we define $C=\V\cap U\setminus U^\perp$, then the induced subgraph on $C$ is complete, because for any $x,y\in C$, we have that $\<x,y\>$ is a tangent with $\<x,y\>\cap U^\perp$ as its isotropic point. As $G$ acts transitively on pairs of vertices of $C$, the parameters $r$ and $\lambda$
    necessary for Theorem \ref{thm:main} exist and we can apply it. More specifically, the point set of the design is $C$,
    while the blocks are
    \[
     B_x = \{ y \in C\colon\, x \text{ and } y \text{ are adjacent} \},
    \]
    where $x \in V(\Gamma) \setminus C$.

    Let us discuss the case of $U$ being a projective plane (a vector space of dimension 3), the smallest interesting case, in more detail. For an anisotropic point $x\notin U$, $W=\<U,x\>$ is a solid with the radical $N$ being a point, a line, or a plane.

 If $N$ is a plane, then $x$ is on $q^2+q+1$ tangents in $W$ and
 hence adjacent to all of $C$. Similarly, if $N$ is a line,
 then $x$ is on $0$ tangents in $W$.

 If $N$ is a point (necessarily, $N \subseteq R$), then $x$ lies on $q\sqrt{q}+q+1$ tangents,
 one of them meeting $U$ in $N$, the other ones corresponding to
 $\sqrt{q}+1$ lines through $N$. These form a $2$-$(q^2, q(\sqrt{q}+1), \sqrt{q})$-design.
\end{example}

\section{Conclusion}

Theorem~\ref{thm:main} yields a framework for finding switching methods for the construction of cospectral graphs, using the language of \((r,\lambda)\)-designs. It unifies many known and new switching methods. Table~\ref{tab:overview} provides an overview of the switching methods that can be derived from this approach and that were mentioned in this paper.

\begin{table}[h]
    \centering
    \begin{tabular}{l|r|r|l|l}
        Method name (if applicable) & \(v\) & Level & Reference & Status \\
        \hline\hline
        GM$_{2m}$ switching & \(2m\) & \(m\) & Theorem~\ref{thm:GM} & known\\
        GM$_{4+\dots+4}$ switching & \(4t\) & \(2\) & Theorem~\ref{thm:GM} & known\\
        GM$_{6+4}$ switching & \(10\) & \(6\) & Example~\ref{ex:GM64} & known\\
        WQH$_{2m}$ switching & \(2m\) & \(m\) & Theorem~\ref{thm:WQH} & known\\
        AH$_{2m}$ switching & \(2m\) & \(2\) & Section~\ref{sec:AH} & known\\
        Fano switching & \(7\) & \(2\) & Example~\ref{ex:fano} & known\\
        Cube switching & \(8\) & \(2\) & Example~\ref{ex:AG32} & known\\
        \(Q(3,3,1)\) switching \cite[Section~4.2]{ABScounting} & \(9\) & \(3\) & Example~\ref{ex:AG23} & known\\
        \hline
        Singer cycle switching & \(q^2+q+1\) & \(q\) & Theorem~\ref{thm:singer} & new\\
        & \(7\) & \(4\) & Example~\ref{ex:new7} & new\\
        & \(8\) & \(3\) & Example~\ref{ex:new8} & new\\
        & \(8\) & \(5\) & Example~\ref{ex:new8b} & new\\
        AG$(2,3)$ switching & \(9\) & \(3\) & Example~\ref{ex:AG23} & new\\
        Paley biplane switching & \(11\) & \(3\) & Example~\ref{ex:paley} & new\\
        $\PG(2,3)$ switching & \(13\) & \(3\) & Example~\ref{ex:PG23} & new\\
    \end{tabular}
    \caption{Switching methods from designs on \(v\) points, mentioned in this paper.}
    \label{tab:overview}
\end{table}

\paragraph{Acknowledgement.}
We thank the reviewers for their very helpful comments.
Robin Simoens is supported by the Research Foundation Flanders (FWO) through the grant 11PG724N.


\end{document}